\documentclass[12pt]{amsart}
\usepackage{amssymb}
\usepackage{fullpage}
\usepackage{amsrefs}
\usepackage{hyperref}

\usepackage{comment}
\theoremstyle{definition}

\newtheorem{theorem}{Theorem}
\newtheorem{lemma}{Lemma}
\newtheorem{proposition}{Proposition}
\newtheorem{example}{Example}

\newcommand{\D}{\mathbb{D}}

\title{Compactness of products and commutators of inner projections}

\author{Peiran Zhang}
\date{\today}                                           % Activate to display a given date or no date
\address{School of Mathematical Sciences,
Dalian University of Technology, Dalian 116024, People's Republic of China}
\email{1196358033@qq.com}

\author{Roumei Tian}
\address{School of Mathematical Sciences,
Dalian University of Technology, Dalian 116024, People's Republic of China}
\email{trm0217@mail.dlut.edu.cn}

\author{Yufeng Lu}
\address{School of Mathematics Sciences, Dalian University of Technology,
Dalian, Liaoning, 116024, P. R. China}
\email{lyfdlut@dlut.edu.cn}

\author{Yixin Yang}
\address{School of Mathematics Sciences, Dalian University of Technology,
Dalian, Liaoning, 116024, P. R. China}
\email{yangyixin@dlut.edu.cn}

\author{Chao Zu*}
\address{School of Mathematical Sciences,
Dalian University of Technology, Dalian 116024, People's Republic of China}
\email{zuchao@dlut.edu.cn}

\thanks{*Corresponding author}

\begin{document}
	\begin{abstract}
		In this paper, we study the compactness of the product and the commutator of two inner projections on the Hardy spaces over the unit disk and the polydisc. For the single-variable case, we provide a complete characterization of the compactness of the commutator of two inner projections by means of Douglas algebra. In the multivariable setting, we discover a rigidity phenomenon: on the bidisc, the product of two inner projections is compact if and only if it has finite rank, whereas on the polydisc of dimension strictly greater than two, any such compact product must be trivial.
	\end{abstract}
	\maketitle		
	
	\section{Introduction}
	Let $\mathbb{D}$ be the open unit disk in the complex plane $\mathbb{C}$, and let $\mathbb{T} = \partial\mathbb{D}$ be its boundary, the unit circle. For an integer $n \ge 1$, the unit polydisc $\mathbb{D}^n$ and its distinguished boundary, the unit torus $\mathbb{T}^n$, are defined as the Cartesian products of $n$ copies of $\mathbb{D}$ and $\mathbb{T}$, respectively. We denote by $L^2(\mathbb{T}^n)$ and $L^\infty(\mathbb{T}^n)$ the standard Lebesgue spaces of square-integrable and essentially bounded functions on the torus, respectively. 
	
	The Hardy space over the polydisc, denoted by $H^2(\mathbb{D}^n)$, consists of all holomorphic functions $f$ on $\mathbb{D}^n$ satisfying
	$$ \sup_{0 \le r_1, \ldots, r_n < 1} \int_{\mathbb{T}^n} |f(r_1w_1, \ldots, r_nw_n)|^2 dm(w) < \infty,
    $$
    where $dm$ denotes the normalized Lebesgue measure on $\mathbb{T}^n$.
	By Fatou's Theorem \cite{Rudin1969}*{Chapter 2}, $H^2(\mathbb{D}^n)$ can be naturally identified with $H^2(\mathbb{T}^n)$, which is the closed subspace of $L^2(\mathbb{T}^n)$ consisting of functions whose Fourier coefficients are supported on $\mathbb{Z}_+^n$ (\(\mathbb{Z}_+\) denotes the set of nonnegative integers). The space $H^\infty(\mathbb{D}^n)$ is the Banach algebra of all bounded holomorphic functions on $\mathbb{D}^n$, equipped with the supremum norm $\|f\|_\infty = \sup_{z \in \mathbb{D}^n} |f(z)|$. A function $\phi \in H^\infty(\mathbb{D}^n)$ is called \textit{inner} if its radial boundary limits satisfy $|\phi(w)| = 1$ almost everywhere on $\mathbb{T}^n$.
	
The \textit{coordinate shifts} on $H^2(\mathbb{D}^n)$ are the tuple of multiplication operators $(M_{z_1}, \ldots, M_{z_n})$, given by $M_{z_j}f(z) = z_j f(z)$ for $1 \le j \le n$. A closed subspace $\mathcal{S} \subset H^2(\mathbb{D}^n)$ is called a \textit{submodule} if it is invariant under each coordinate shift, i.e., $M_{z_j}\mathcal{S} \subset \mathcal{S}$ for all $j$. Correspondingly, a closed subspace $\mathcal{Q}$ is said to be a \textit{quotient module} if  $M_{z_j}^*\mathcal{Q} \subset \mathcal{Q}$ for all \(j\).
	  The classical Beurling's theorem states that every non-trivial submodule of $H^2(\mathbb{D})$ takes the form $\mathcal{S} = \theta H^2(\mathbb{D})$, where $\theta$ is an inner function. However, in higher dimensions ($n > 1$), the situation becomes significantly more intricate, as there exist submodules that are not generated by inner functions (see \cite{Rudin1969}*{Chapter 5}). In this paper, we refer to submodules specifically generated by inner functions as \textit{Beurling-type submodules}, and their corresponding orthogonal complements as \textit{Beurling-type quotient modules}.

	An orthogonal projection $P$ on some Hilbert space $\mathcal{H}$ is a bounded self-adjoint operator satisfying
	$
	P^2=P.
	$
	In general operator theory, the geometric relationship between two closed subspaces of a Hilbert space can be effectively captured by the algebraic properties of their corresponding orthogonal projections $P$ and $Q$.  This approach, often referred to as the two projections theory, was pioneered by Dixmier \cite{Dixmier1948} and Halmos \cite{Halmos1969}. For instance, the spectral properties of the product $PQ$ are intimately related to the minimal angle between the two subspaces \cite{Davis1958}. Furthermore, the commutator  $[P, Q]$ provides crucial information about their relative position \cite{Bottcher2010}. Motivated by this perspective, we investigate analogical problems within  $H^2(\mathbb{D}^n)$.

	An orthogonal projection $P_{\mathcal{Q}_\theta}$ on $H^2(\mathbb{D}^n)$ is called an \textit{inner projection} if it is an orthogonal projection onto a Beurling-type quotient module generated by the inner function $\theta$, that is,
	$$
	\operatorname{Ran}P_{\mathcal{Q}_\theta}=H^2(\mathbb{D}^n)\ominus\theta H^2(\mathbb{D}^n).
	$$
	Correspondingly, the orthogonal projection onto the Beurling-type submodule $\theta H^2(\mathbb{D}^n)$ is denoted by
	$P_{\theta}:=I-P_{\mathcal{Q}_\theta}.$

	In this paper, we first focus on the compactness of the commutator of inner projections on \(H^2(\mathbb{D})\). Debnath et al. \cite{Debnath2024} provided a complete characterization when the commutators are zero on $H^2(\mathbb{D}^n)$ (see \cite{Bickel2016} for the original question). Our first main result extends their result and provides a characterization of the compact commutator $[P_{\mathcal{Q}_{\varphi}}, P_{\mathcal{Q}_{\psi}}]$ on  $H^2(\mathbb{D})$, based on Douglas's localization. Furthermore, we provide several examples of obtaining more specific characterizations for certain special inner functions.

	Furthermore, we investigate the compactness of products of inner projections in the multivariable setting. It is well known that on the Hardy space $H^2(\mathbb{D})$,  the compactness of the product $P_{\mathcal{Q}_{\varphi}} P_{\mathcal{Q}_{\psi}}$ is equivalent to the compactness of the semi-commutators of corresponding Toeplitz operators \cite{Berger1978}. 
    In 1978, Axler, Chang, and Sarason \cite{ACS1978} characterized the compactness of such semi-commutators by means of Douglas algebra. We address the analogous problem on the polydisc $H^2(\mathbb{D}^n)$. We show that on the bidisc ($n=2$), any compact product must be of finite rank, which happens exactly when the inner symbols are separated finite Blaschke products. In contrast, for the polydisc with $n > 2$, a compact product must be identically zero.

    \quad\\
    \textbf{Acknowledgments}  We thank Rounak Biswas and Srijan Sarkar for correcting a typo in our definition of the defect operator. We also notice that their recent work \cite{BiswasSarkar26} investigates analogous questions and proves the same result via the Sz.-Nagy–Foiaș model, offering a different method from ours.

\section{\texorpdfstring{Commutators  of inner projections on $H^2(\mathbb{D})$}{Commutators}}
\subsection{Preliminaries}
Recall that $L^2(\mathbb{T})$ denotes the Lebesgue space of square integrable functions on the unit circle $\mathbb{T}$. The Hardy space $H^2(\mathbb{T})$ is the subspace of $L^2(\mathbb{T})$ consisting of functions with vanishing negative Fourier coefficients. Let $P$ denote the orthogonal projection from $L^2(\mathbb{T})$ onto $H^2(\mathbb{T})$. For $f \in L^\infty(\mathbb{T})$, the Toeplitz operator $T_f$ and the Hankel operator $H_f$ are defined on $H^2(\mathbb{T})$ by
$$T_f h = P(fh) \quad \text{and} \quad H_f h = P U(fh),$$
where $U$ is the unitary operator defined by $U h(z) = \overline{z}\overline{h}(z)$. Clearly, \(H^*_f=H_{f^*}\), where \(f(z)^*=\overline{f}(\bar{z})\). A fundamental algebraic relation connecting these operators is that for any $f, g \in L^\infty(\mathbb{T})$, 
$$T_{fg} = T_f T_g + H_{\tilde{f}} H_g,$$
where \(\tilde{f}(z)=f(\bar{z})\).

To rigorously state the local conditions for the compactness of operator products, we must introduce the concepts of uniform algebras and their maximal ideal spaces. Let $C(X)$ denote the Banach algebra of all complex-valued continuous functions on some compact Hausdorff space $X$ equipped with the supremum norm. A \textit{uniform algebra} $A$ on $X$ is a uniformly closed subalgebra of $C(X)$ that contains the constant functions and separates the points of $X$ \cite{Douglas1973}*{Chapter 6}. The  \textit{maximal ideal space} of \(A\), denoted by $M(A)$, is the space of all non-zero multiplicative linear functionals on $A$, endowed with the Gelfand (weak-*) topology. Under this topology, $M(A)$ is a compact Hausdorff space \cite{Garnett1981}*{Chapter 10}.

In our context, $H^\infty(\mathbb{D})$ and $H^\infty + C$ (the closed linear span of $H^\infty$ and $C(\mathbb{T})$) are both uniform algebras. By Carleson's Corona Theorem, the open disk $\mathbb{D}$ can be viewed as a dense subset of $M(H^\infty)$ \cite{Garnett1981}*{Chapter 8}. Furthermore, the maximal ideal space of $H^\infty + C$ can be canonically identified with $M(H^\infty) \setminus \mathbb{D}$. The \v{S}ilov boundary of \(M(H^\infty)\) is \(M(L^\infty)\). A subset $S$ of $M(L^\infty)$ is called a \textit{support set} if it is the support of a representing measure for a functional in $M(H^\infty+C)$.

To study the commutators of inner projections in the single-variable case, our primary strategy is to decompose them as products of Hankel and Toeplitz operators. The following characterization by Chu \cite{Chu2015} serves as the fundamental tool in our proof.

\begin{theorem}[\cite{Chu2015} Theorem 1.5 \& Theorem 1.6]\label{thm_chu}
	Let \( f, g \in L^\infty(\mathbb{T})\). Then the following are equivalent:
	\begin{enumerate}
		\item the product \( K = H_f T_g \) of the Hankel operator \( H_f \) and the Toeplitz operator \( T_g \) is compact;
		\item \( H^\infty [f] \cap H^\infty [g, fg] \subset H^\infty + C \);
		\item for each support set \( S \), one of the following holds:
		\begin{enumerate}
			\item \( f|_S \in H^\infty |_S \);
			\item \( g|_S \in H^\infty |_S \) and \( (fg)|_S \in H^\infty |_S \).
		\end{enumerate}
	\end{enumerate}
\end{theorem}

\subsection{Main result}
Our main result characterizes the compactness of the commutators of two inner projections on the Hardy space over the unit disk.

\begin{theorem}\label{thm_A}
	Let \(\varphi\) and \(\psi\) be two inner functions in \(H^\infty(\D)\), \(P_{\mathcal{Q}_\varphi}\) and \(P_{\mathcal{Q}_\psi}\) be the corresponding inner projections generated by \(\varphi\) and \(\psi\) respectively. The following statements are equivalent:
	\begin{enumerate}
		\item \([P_{\mathcal{Q}_\varphi},P_{\mathcal{Q}_\psi}]\) is compact;
		\item \(H^\infty[\overline{\varphi}\psi]\cap H^\infty[\varphi\overline{\psi}]\subset H^\infty+C\);
		\item for each support set \(S\), either \(\varphi\overline{\psi}|_S\in H^\infty|_S\) or \(\overline{\varphi}\psi|_S\in H^\infty|_S\).
	\end{enumerate}
\end{theorem}

\begin{proof}
It is easy to see that \([P_{\mathcal{Q}_{\varphi}},P_{\mathcal{Q}_{\psi}}]=[P_{\varphi},P_{\psi}].\)
Let $K:= [P_{\varphi}, P_{\psi}] = P_{\varphi} P_{\psi} - P_{\psi} P_{\varphi}$. 
With the decomposition \(H^2(\mathbb{D})=\varphi H^2(\mathbb{D})\oplus\mathcal{Q}_{\varphi}\), write $K = P_{\varphi} P_{\psi} P_{\varphi}^\perp - P_{\varphi}^\perp P_{\psi} P_{\varphi}$, where $P_{\varphi}^\perp = I - P_{\varphi}$. Notice that \(P_{\varphi} P_{\psi} P_{\varphi}^\perp\) is a block upper triangular matrix and 
\[
(P_{\varphi}^\perp P_{\psi} P_{\varphi})^*=P_{\varphi} P_{\psi} P_{\varphi}^\perp.
\]
Hence, $K$ is compact if and only if $P_{\varphi}^\perp P_{\psi} P_{\varphi}$ is compact. Let \(A:=P_\varphi^\perp P_\psi T_\varphi\), if \(AT_\psi^*\) is compact. Then
\[
	AA^*=T^*_\varphi P_\psi P^\perp_\varphi P_\psi T_\varphi=(A^*T_\psi)T^*_\psi T_\varphi
	\]
is compact. This implies that \(A\) is compact. Since $P_\varphi=T_\varphi T_{\varphi}^*$, we have
\[
	A=AT_\varphi^*T_\varphi=P_{\varphi}^\perp P_{\psi} P_{\varphi}T_\varphi.
	\]
Then \(P_{\varphi}^\perp P_{\psi} P_{\varphi}\) is compact implies \(A\) is compact. Therefore, \(P_\varphi^\perp P_\psi P_\varphi\) is compact if and only if \(A^*T_\psi\) is compact.

Following a direct computation,
\[
\begin{aligned}
	A^*T_\psi &=T_{\varphi}^* P_{\psi} P_{\varphi}^\perp T_{\psi}\\ &= T_{\varphi}^* T_{\psi} T_{\psi}^* (I - T_{\varphi} T_{\varphi}^*) T_{\psi} \\
	&= T_{\varphi}^* T_{\psi} - (T_{\varphi}^* T_{\psi})(T_{\psi}^* T_{\varphi})(T_{\varphi}^* T_{\psi}) \\
	&= T_{\overline{\varphi}\psi} - T_{\overline{\varphi}\psi} T_{\varphi\overline{\psi}} T_{\overline{\varphi}\psi}.
\end{aligned}
\]

For simplicity, set \(\phi:=\overline{\varphi}\psi\). Since \(|\phi|=1\) on \(\mathbb{T}\) almost everywhere, we have
\[
T_{\overline{\phi}} T_\phi+ H_{\overline{\tilde{\phi}}}H_\phi=T^*_\phi T_\phi+ H^*_\phi H_\phi=I,
\] 
and then
\[
T_\phi(I-T_\phi^*T_\phi)=T_\phi H_\phi^*H_\phi.
\]
We claim that \(T_\phi H_\phi^*H_\phi\) is compact if and only if \(H_\phi T_\phi^*\) is compact. In fact, for any \(f\in L^\infty(\mathbb{T})\), the compactness of \(T_f H_f^*H_f\) implies that \(T_f H_f^*H_fT_f^*\) is compact. Since \(A^*A\) is compact if and only if \(A\) is compact, we obtain that \(T_fH_f^*\) is compact.

We can conclude the equivalence of items (1), (2), and (3) directly by Theorem \ref{thm_chu}. Substituting $\phi = \overline{\varphi}\psi$ and noting that $\overline{\phi}\phi = |\phi|^2 = 1$ a.e. on $\mathbb{T}$, this is exactly 
$$
H^\infty[\overline{\varphi}\psi] \cap H^\infty[\varphi\overline{\psi}]=H^\infty[\overline{\varphi}\psi] \cap H^\infty[\varphi\overline{\psi}] \subset H^\infty + C.
$$ 
Since $1\in H^\infty|_S$ for each support set $S$. Then for each support set $S$, either $\phi|_S \in H^\infty|_S$ or $\overline{\phi}|_S \in H^\infty|_S$.
\end{proof}

\subsection{Further Discussion of Theorem \ref{thm_A}}
One expects the class of inner-function pairs with compact projection commutators to be strictly larger than that with compact projection products. A trivial source of such examples arises from a common inner factor: if \(\varphi\) and \(\psi\) share a nontrivial inner factor \(\theta\) that is not a finite Blaschke product, then \([P_{\mathcal{Q}_{\varphi}}, P_{\mathcal{Q}_{\psi}}]\) is compact, whereas \(P_{\mathcal{Q}_{\varphi}} P_{\mathcal{Q}_{\psi}}\) fails to be compact. It is therefore natural to ask whether there exist more general, nontrivial examples beyond this common factor construction.
Before presenting such a nontrivial example, we need some definitions and well-known results.

A sequence \(\{z_n\}_{n=1}^\infty\) in the open unit disk \(\mathbb{D}\) is called an \textit{interpolating sequence} for \(H^\infty(\mathbb{D})\) if for every bounded sequence \(\{w_n\}_{n=1}^\infty \in \ell^\infty\), there exists a bounded analytic function \(f \in H^\infty(\mathbb{D})\) such that
\[
f(z_n) = w_n \quad \text{for all } n = 1,2,\dots.
\]
The discrete positive measure $\mu$ induced by $\{z_n\}_{n=1}^\infty$, defined as $\mu = \sum_{n=1}^\infty (1 - |z_n|^2) \delta_{z_n}$, where $\delta_{z_n}$ is the Dirac measure at $z_n$, is called a \textit{Carleson measure} if there exists a constant $C > 0$ such that $\mu(S(I)) \le C|I|$ for every subarc $I \subset \mathbb{T}$. Here, $|I|$ denotes the normalized arc length of $I$, and $S(I) = \left\{ z \in \mathbb{D} : z/|z| \in I, 1 - |z| \le |I| \right\}$. Carleson's interpolation theorem \cite{Garnett1981}*{Chapter 7} states that a sequence $\{z_n\}_{n=1}^\infty \subset \mathbb{D}$ is an $H^\infty$ interpolation sequence if and only if there exists a constant $\delta>0$ such that for every $n$,
\[
\prod_{k \neq n} \rho(z_n, z_k) \ge \delta,
\]
where $\rho(z,w)$ denotes the pseudo-hyperbolic distance on $\mathbb{D}$. Equivalently, the sequence is uniformly separated (i.e., $\inf_{i\neq j}\rho(z_i,z_j)>0$) and the measure $\mu = \sum_{n=1}^\infty (1-|z_n|^2)\delta_{z_n}$ is a Carleson measure.

A Blaschke product is called an \textit{interpolating Blaschke product} if it has
distinct zeros and if these zeros form an interpolating sequence.

The following theorem gives the characterization for the compactness of \(P_{\mathcal{Q}_{\varphi}}P_{\mathcal{Q}_{\psi}}\).
\begin{theorem}[\cite{ACS1978} Axler, Chang and Sarason]\label{thm_ACS}
    Let \(\varphi\) and \(\psi\) are two inner functions. Then the following are equivalent:
    \begin{enumerate}
        \item \(P_{\mathcal{Q}_{\varphi}}P_{\mathcal{Q}_{\psi}}\) is compact;
        \item \(H^\infty[\overline{\varphi}]\cap H^\infty[\overline{\psi}]\subseteq H^\infty+C\);
        \item for each support set \(S\), either \(\overline{\varphi}|_S\in H^\infty|_S\) or \(\overline{\psi}|_S\in H^\infty|_S\);
        \item \(\lim_{|z|\to 1}\max \{|\varphi(z)|,|\psi(z)|\}=1\).
    \end{enumerate}
\end{theorem}

We now present an example of two relatively prime inner functions for which the commutator of the corresponding inner projections is compact, whereas their product is not compact.

\begin{example}\label{exm_1}
	Let \(B_1\) and \(B_2\) be two infinite Blaschke products with zero sequences 
	$$ a_n = 1 - 2^{-n}, \quad b_n = 1 - 2^{-n}\left(1 + \frac{1}{n}\right),\qquad n\geq 1$$
	respectively. Then we claim that
	\begin{itemize}
		\item[1.] the product \(P_{B_1}P_{B_2}\) is not compact;
		\item[2.] the commutator \([P_{B_1},P_{B_2}]\) is compact.
	\end{itemize}
\end{example}
Clearly, $B_1$ and $B_2$ have no common zeros and are strictly relatively prime.  Moreover, we have
$$ \lim_{n \to \infty} \rho(a_n, b_n) = \lim_{n \to \infty} \frac{\frac{1}{n}}{2 + \frac{1}{n}} = 0. $$
Obviously, $\lim_{n\to\infty}B_1(a_n) = 0$, and the Schwarz-Pick lemma implies that $|B_2(a_n)| \le \rho(a_n, b_n)$. Therefore,
$$ \lim_{n\to\infty} \max(|B_1(a_n)|, |B_2(a_n)|) = \lim_{n\to\infty} |B_2(a_n)| = 0 \neq 1, $$
and \(P_{B_1}P_{B_2}\) is not compact by condition (4) in Theorem \ref{thm_ACS}.

Since both \(B_1\) and \(B_2\) are interpolating Blaschke products with only one accumulate point at \(1\), and \(\lim_{n\to\infty}\rho(a_n,b_n)=0\). The zeros of \(B_1\) can be considered as a small perturbation of zeros of \(B_2\), then \(B_2/B_1\in H^\infty+C\)(see \cite{Guillory1981} for details). Therefore, \([P_{B_1},P_{B_2}]\) is compact by condition (2) in Theorem \ref{thm_A}.

For a pair of relatively prime inner functions, the condition (4) in Theorem \ref{thm_ACS} implies Carleson's Corona condition (see \cite{Garnett1981} for the details of Carleson's Corona theorem). The following proposition clarifies their precise relationship.

\begin{proposition} \label{prop_1}
	Let $\varphi$ and $\psi$ be relatively prime inner functions in $H^\infty(\mathbb{D})$. Conditions (S) and (C) are defined as:
	\begin{itemize}
		\item[(S)] $\lim_{|z| \to 1} \max\{|\varphi(z)|, |\psi(z)|\} = 1$;
		\item[(C)] $\inf_{z \in \mathbb{D}} (|\varphi(z)| + |\psi(z)|) > 0$.
	\end{itemize}
	Then (S) implies (C), but the converse is false.
\end{proposition}

\begin{proof}
	(S) $\Rightarrow$ (C): By $(S)$, there exists $r \in (0,1)$ such that $\max\{|\varphi(z)|, |\psi(z)|\} > 1/2$ for $r < |z| < 1$. Thus $|\varphi| + |\psi| > 1/2$ in this annulus. On the compact disk $\{|z| \le r\}$, since $\varphi$ and $\psi$ are relatively prime, they have no common zeros. By continuity, $|\varphi| + |\psi|$ attains a positive minimum $\delta > 0$. Taking the minimum of $\{1/2, \delta\}$ yields the global lower bound for (C).
	
	(C) $\nRightarrow$ (S): Consider the interpolating sequences $a_n = 1 - 4^{-n}$ and $b_n = 1 - 2 \cdot 4^{-n}$. Let \(B_1\) and \(B_2\) be the corresponding Blaschke products. It is easy to check that \(\{a_n\}\cup\{b_n\}\) is also an interpolating sequence, which implies that (C) holds for $B_1$ and $B_2$. However, 
	$$ \lim_{n \to \infty} \max\{|B_1(a_n)|, |B_2(a_n)|\} = \lim_{n \to \infty} |B_2(a_n)| \le \lim_{n \to \infty} \rho(a_n, b_n) = 1/3 \neq 1. $$
	Thus (S) fails.
\end{proof}

We will later show that, for relatively prime inner functions, compactness of projection products and of commutators are equivalent under condition $(C)$. Moreover, this equivalence extends to general inner functions under a strictly weaker boundary condition.

Let $(WC)$ be the weaker Carleson's condition defined by:
$$(WC)\quad \liminf_{|z| \to 1} (|\varphi(z)| + |\psi(z)|) > 0. $$

\begin{proposition} 
Let $\varphi$ and $\psi$ be two inner functions satisfying condition $(WC)$. Then the commutator $[P_{\varphi}, P_{\psi}]$ is compact if and only if the projection product $P_{\mathcal{Q}_\varphi} P_{\mathcal{Q}_\psi}$ is compact.
\end{proposition}

\begin{proof}
Assume that the commutator $[P_{\varphi}, P_{\psi}]$ is compact. By Theorem \ref{thm_A}, for every support set $S$ of $M(H^\infty + C)$, we have either $\overline{\varphi}\psi|_S \in H^\infty|_S$ or $\varphi\overline{\psi}|_S \in H^\infty|_S$.

If $\varphi$ and $\psi$ satisfy condition $(WC)$, there exists an annulus $r < |z| < 1$ and a constant $\delta_0 > 0$ such that 
$$ \inf_{r < |z| < 1} (|\varphi(z)| + |\psi(z)|) > \delta_0. $$
Since the functions are bounded away from zero near the boundary, their zero sets can only intersect in the compact disk $\{|z| \le r\}$. %meaning they share at most finitely many common zeros. Consequently,
Hence, 
$H^2(\mathbb{D}) / (\varphi H^2(\mathbb{D}) + \psi H^2(\mathbb{D}))$ is finite-dimensional, which implies the existence of a finite Blaschke product $B$ (whose zeros correspond to the common zeros of $\varphi$ and $\psi$) such that $B H^2(\mathbb{D}) = \varphi H^2(\mathbb{D}) + \psi H^2(\mathbb{D})$. Since on the unit disk the \(H^2\)-corona theorem is equivalent to the \(H^\infty\)-corona theorem, it follows that, there are \(u,v\in H^\infty\) such that
$$ u\varphi + v\psi = B.$$
We restrict this identity to an arbitrary support set $S \subset M(L^\infty)$ to get
\begin{equation}\label{eq_1}
    u|_S\varphi|_S + v|_S\psi|_S = B|_S.
\end{equation}
Since $S$ is contained in the \v{S}ilov boundary, where inner functions are unimodular. Multiplying the identity \eqref{eq_1} by $\overline{\varphi}|_S$, we obtain
$$ \overline{\varphi}B|_S = u|_S(\varphi\overline{\varphi})|_S + v|_S(\overline{\varphi}\psi|_S) = u|_S + v|_S(\overline{\varphi}\psi|_S). $$
If $\overline{\varphi}\psi|_S \in H^\infty|_S$, since $u|_S, v|_S \in H^\infty|_S$, the algebraic closeness of $H^\infty|_S$ guarantees that $\overline{\varphi}B|_S \in H^\infty|_S$. Notice that $\overline{B}|_S \in H^\infty|_S$, since $B$ is a finite Blaschke product. Multiplying \(\overline{\varphi}B|_S\) by $\overline{B}|_S$, we have
$$ \overline{B}|_S (\overline{\varphi}B|_S) = \overline{\varphi}|_S (\overline{B}B)|_S = \overline{\varphi}|_S \in H^\infty|_S. $$

By means of symmetric reasoning, if $\varphi\overline{\psi}|_S \in H^\infty|_S$, it follows that $\overline{\psi}|_S \in H^\infty|_S$.

Therefore, either $\overline{\varphi}|_S \in H^\infty|_S$ or $\overline{\psi}|_S \in H^\infty|_S$ holds for any support set $S$. By condition (3) in Theorem \ref{thm_ACS} , this forces the product $P_{\mathcal{Q}_\varphi} P_{\mathcal{Q}_\psi}$ to be compact.
\end{proof}

%When considering the condition \((WC)\) for two interpolating Blaschke products, we can provide more details on the separation condition.

For interpolated Blaschke products, we can characterize the (WC) condition by the separability of their zeros under the pseudo-hyperbolic metric.

\begin{proposition} 
	Let $B_1$ and $B_2$ be two interpolating Blaschke products with zero sequences $Z_1 = \{a_n\}$ and $Z_2 = \{b_n\}$, respectively. Then $B_1$ and $B_2$ satisfy condition $(WC)$
	if and only if there exist $r \in (0, 1)$ and $\delta > 0$ such that
    \begin{equation}\label{sc}
        \inf \left\{ \rho(\alpha, \beta): \alpha\in Z_1, \beta\in Z_2, ,|\alpha|,|\beta|> r \right\} \ge \delta > 0,
    \end{equation}
	where $\rho(z, w)$ is the pseudo-hyperbolic distance on \(\mathbb{D}\).
\end{proposition}

\begin{proof}
    $(\Rightarrow)$ Assume $B_1$ and $B_2$ satisfy condition $(WC)$. By definition, there exist $r \in (0, 1)$ and $\delta_0 > 0$ such that for all $z$ in the annulus $r < |z| < 1$, 
	$$ |B_1(z)| + |B_2(z)| \ge \delta_0. $$
	Take any $\alpha \in Z_1$ with $|\alpha| > r$. Since $\alpha$ is a zero of $B_1$, then
	$$ |B_2(\alpha)|=|B_2(\alpha)-B_1(\alpha))| \ge \delta_0. $$
	By Schwarz-Pick lemma, for any $\beta \in Z_2$, we have
	$$
    \rho(\alpha,\beta) \geq |B_2(\alpha)| \geq \delta_0. 
    $$
    $(\Leftarrow)$ Conversely, assume there exists $r \in (0,1)$ such that points in the sequences $Z_1^{r} = \{\alpha\in Z_1: |\alpha| > r\}$ and $Z_2^{r} = \{\beta\in Z_2: |\beta| > r\}$ are satisfy condition \eqref{sc}. Then there exists a constant \(\delta_0>0\) such that for any $z,w\in Z_1^{r} \cup Z_2^{r}$, \(\inf\rho(z,w)\geq\delta_0>0.\) Let \(\mu_1\) and \(\mu_2\) be the corresponding Carleson measure induced by \(Z_1^r\) and \(Z_2^r\) and \(\mu\) be the discrete measure induced by \(Z_1^{r} \cup Z_2^{r}\). Then \(\mu=\mu_1+\mu_2\), which is also a Carleson measure. Hence, $Z_1^r \cup Z_2^r$ is an interpolating sequence. Therefore, the corresponding interpolating Blaschke products \(B_1'\) and \(B_2'\)
    satisfy
	$$ \inf_{z \in \mathbb{D}} \left( |B_1'(z)| + |B_2'(z)| \right) \geq \eta > 0 $$
    for some constant \(\eta>0\).
	
	Since \(B_1/B_1'\) and \(B_2/B_2'\) are finite Blaschke products, then their zero points are contained in \(\{|z|\leq r\}\). This confirms that $B_1$ and $B_2$ satisfy condition $(WC)$, completing the proof.
\end{proof}

Finally, we present a class of pairs of interpolating Blaschke products. Their zero set fails the condition \eqref{sc}. For these pairs, the commutators of the corresponding inner projections are compact. The Example \ref{exm_1} can be regarded as a direct consequence of this result.

For $m_1, m_2$ in the maximal ideal space $M(H^\infty)$, the pseudo-hyperbolic distance is defined as
\[
    \rho(m_1, m_2) = \sup \{ |\hat{f}(m_2)| : f \in H^\infty, \|f\|_\infty \le 1, \hat{f}(m_1) = 0 \}.
\]
The \textit{Gleason part} containing $m$, denoted by $P(m)$, is the equivalence class of $m$ under the relation $\rho < 1$, i.e., $P(m) = \{ x \in M(H^\infty) : \rho(x, m) < 1 \}$. If $P(m)$ is a non-trivial Gleason part (i.e., containing more than one point), Hoffman's theorem states that $P(m)$ has the structure of an analytic disc \cite{Garnett1981}*{Chapter 10}. Specifically, there exists a bijective analytic map $L_m : \mathbb{D} \to P(m)$ with $L_m(0) = m$. This map is obtained as a pointwise limit of Mobius transformations $L_{\alpha_i}(z) = \frac{z + \alpha_i}{1 + \bar{\alpha}_i z}$, where $\{\alpha_i\}$ (e.g. an interpolating sequence) is a net in $\mathbb{D}$ converging to $m$ in the Gelfand topology.

\begin{proposition}
Let $B_1$ and $B_2$ be two interpolating Blaschke products with zero sequences $\{\alpha_n\}_{n=1}^\infty$ and $\{\beta_n\}_{n=1}^\infty$, respectively. If
    \[
        \lim_{n \to \infty} \rho(\alpha_n, \beta_n) = 0,
    \]
    then the commutator $[P_{B_1}, P_{B_2}]$ is compact.
\end{proposition}
\begin{proof}
    To prove the commutator $[P_{B_1}, P_{B_2}]$ is compact, by Theorem \ref{thm_A}, it is enough to prove $\bar{B}_1 B_2\in H^\infty + C$.This is equivalent to proving that the $B_2 / B_1$ is analytic on every Gleason part $P$ contained in the corona $M(H^\infty) \setminus \mathbb{D}$. Let $Z_1$ and $Z_2$ denote the zero sequences $\{\alpha_n\}$ and $\{\beta_n\}$, respectively.

    Let $P$ be an arbitrary Gleason part in $M(H^\infty) \setminus \mathbb{D}$. We consider two cases based on the intersection of $P$ with the closure of the zero sequence $\overline{Z_1}$ in the Gelfand topology.

    \textbf{Case 1:} $P \cap \overline{Z_1} = \emptyset$. 
    By Hoffman's result in \cite{Hoffman1967}, since the part $P$ avoids the closure of the interpolating sequence $Z_1$, the Gelfand transform $\hat{B}_1$ is bounded away from zero on $P$. Then there exists $\eta > 0$ such that $|\hat{B}_1(x)| \ge \eta > 0$ for all $x \in P$. Consequently, $1/\hat{B}_1$ is bounded and analytic on $P$. Since $\hat{B}_2$ is analytic on $P$, the ratio $B_2/B_1$ is analytic on $P$.

    \textbf{Case 2:} $P \cap \overline{Z_1} \neq \emptyset$. 
    In this case, $P$ is always a non-trivial part (an analytic disc). Let $m \in P \cap \overline{Z_1}$, which implies $\hat{B}_1(m) = 0$. Since $\lim_{n \to \infty} \rho(\alpha_n, \beta_n) = 0$, then there are two subsequences \(\{\alpha_{n_k}\}\) and \(\{b_{n_k}\}\) such that \(\alpha_{n_k}\to m\) and \(b_{n_k}\to m\) in \(M(H^\infty+C)\). This implies that $\overline{Z_1} \setminus \mathbb{D} = \overline{Z_2} \setminus \mathbb{D}$. Therefore, we have $\hat{B}_2(m) = 0$.

    Let $L_m : \mathbb{D} \to P$ be the analytic map with $L_m(0) = m$. The compositions $F_1 = \hat{B}_1 \circ L_m$ and $F_2 = \hat{B}_2 \circ L_m$ are interpolating Blaschke products on $\mathbb{D}$. Since $Z_1$ and $Z_2$ are interpolating sequences, there exist separation constants $\delta_1, \delta_2 > 0$ such that
    \[
        \prod_{j \neq i} \rho(\alpha_i, \alpha_j) \ge \delta_1 \quad \text{and} \quad \prod_{j \neq i} \rho(\beta_i, \beta_j) \ge \delta_2.
    \]
    Let $\delta = \min\{\delta_1, \delta_2\}$. Since pseudo-hyperbolic distances are preserved under $L_m$, then any two distinct zeros of $F_1$ (or $F_2$) are separated by a pseudo-hyperbolic distance of at least $\delta$. 

    Since $F_1(0) = F_2(0) = 0$, it follows that in the open disk $\mathbb{D}_\delta = \{z \in \mathbb{D} : |z| < \delta\}$, both $F_1$ and $F_2$ have exactly one simple zero located at the origin. Thus, the singularity at $z=0$ for the ratio $F_2(z)/F_1(z)$ is removable, making $F_2/F_1$ analytic on $\mathbb{D}_\delta$. 

    If $F_1$ has another zero $w \in \mathbb{D}$ (with $|w| \ge \delta$), we can select the maximal ideal $m' = L_m(w) \in P$ as the new base point. Replacing $L_m$ with the corresponding map $L_{m'}$ applies the same argument, then the \((B_2\circ L_{m'})/(B_1\circ L_{m'})\) is analytic in a $\delta$-neighborhood of $w$. Hence,
    %$F_2/F_1$ is globally analytic on $\mathbb{D}$, which confirms that
    $B_2/B_1$ is analytic on the entire Gleason part $P$.

    Thus, we conclude that $\bar{B}_1 B_2 \in H^\infty + C$.
    % \[
    % H[\bar{B}_1 B_2]\cap H[\bar{B}_2 B_1]\subseteq H^\infty + C.
    % \]
    By Theorem \ref{thm_A}, the proof is complete.
\end{proof}
\subsection{Remark}
As illustrated by the above discussion, the compactness of the commutator $[P_{\mathcal{Q}_{\varphi_1}}, P_{\mathcal{Q}_{\varphi_2}}]$ is closely related to the condition $\varphi_1/\varphi_2 \in H^\infty+C$, although the two are not strictly equivalent. Characterizing this divisibility purely in terms of the boundary behavior of the zeros is not easy. While Guillory and Sarason  \cites{Guillory1981, Guillory1984} investigated specific divisibility properties, and Izuchi \cite{Izuchi2002} later extended this to singular inner functions, the question of when an arbitrary function in $H^\infty+C$ admits a non-trivial inner divisor remains an open problem. Returning to the commutator itself, we pose the following question:
\begin{center}
Can the compactness of $[P_{\mathcal{Q}_{\varphi}}, P_{\mathcal{Q}_{\psi}}]$ be characterized entirely by the boundary behavior of the inner symbols $\varphi$ and $\psi$?
\end{center}

\section{\texorpdfstring{Products of inner projections on \(H^2(\mathbb{D}^n)\)}{Products}}
\subsection{Preliminaries}
In this section, we address the question of characterizing the compactness of the product of two inner projections $P_{\mathcal{Q}_\varphi} P_{\mathcal{Q}_\psi}$ on \(H^2(\mathbb{D}^n)\). It is well known that on the Hardy space $H^2(\mathbb{D})$, the compactness of the product $P_{\mathcal{Q}_\varphi} P_{\mathcal{Q}_\psi}$ is equivalent to the compactness of the commutator of corresponding Toeplitz operators $[T^{*}_{\varphi}, T_{\psi}]$ (see \cite{ACS1978}). However, it does not hold in the multivariate setting. For instance, let \(\varphi\) and \(\psi\) be infinite Blaschke products which only depends on \(z_1\) and \(z_2\) respectively, then $P_{\mathcal{Q}_{\varphi}} P_{\mathcal{Q}_\psi}$ is no longer compact, whereas $\left[T^*_{\varphi}, T_\psi\right]=0$. In fact, the compactness of $\left[T^*_{\varphi}, T_\psi\right]$ is equivalent to the inner function symbols being separated in variables. 

\begin{theorem}[\cite{Ding2003} Theorem 2.2] \label{ding}
	Let $f$ and $g$ be bounded pluriharmonic functions on $\mathbb{D}^n$ ($n > 1$). The following conditions are equivalent:
	\begin{enumerate}
		\item[(1)] $T_f T_g - T_g T_f = 0$;
		\item[(2)] $T_f T_g - T_g T_f$ is compact;
		\item[(3)] For each $j$, either $f$ and $g$ are analytic with respect to $z_j$, or $f$ and $g$ are conjugate analytic with respect to $z_j$, or there exists a constant $C_j$ such that $g - C_j f$ is constant with respect to $z_j$.
	\end{enumerate}
\end{theorem}

Although the equivalence between the compact product $P_{\mathcal{Q}\varphi} P_{\mathcal{Q}_\psi}$ and the compact commutator $[T^{*}_{\varphi}, T_{\psi}]$ fails on the Hardy space $H^2(\mathbb{D}^n)$ ($n>1$), we find that the compactness of the defect operators of $T_{\varphi}$ and $T_{\psi}$ is equivalent to that of the product $P_{\mathcal{Q}_\varphi} P_{\mathcal{Q}_\psi}$.

The defect operator for canonical isometric pair $T=(T_1,T_2)$ on some Hilbert space $\mathcal{H}$ is defined by
\[
\Delta_T:=I-T_1T_1^*-T_2T_2^*+T_1T_1^*T_2T_2^*.
\]
The literature already contains a variety of results concerning defect operators of analytic Hardy submodules. As is shown in \cite{GuoYang} \cite{Guo2004} \cite{Yang2005} \cite{YangRW}, this operator is crucially related to operator theory and the structure of submodules.

In what follows, we require information about the defect operators and the relative position relationship for two closed subspaces. We collect the needed information here, which is useful in our proof.

\begin{lemma}[\cite{GuoWang2007} Proposition 1.1]\label{lem_1}
	For any isometric pair $T=\left(T_1, T_2\right)$ on $\mathcal{H}$,\par
    \begin{enumerate}
        \item $\Delta_T$ is finite rank if and only if both $\left[T_1^*, T_1\right]\left[T_2^*, T_2\right]$ and $\left[T_1^*, T_2\right]$ are finite rank; 
        \item $\Delta_T$ is compact if and only if $\left[T_1^*, T_1\right]\left[T_2^*, T_2\right]$ and $\left[T_1^*, T_2\right]$ are both compact.
    \end{enumerate}

\end{lemma}

\begin{lemma}[\cite{GuoWang20072} Proposition 3.2]\label{thm_3}
	Let $\mathcal{H}$ be a Hilbert space, $N_1$ and $N_2$ be two closed subspaces of $\mathcal{H}$. Then $P_{N_1} P_{N_2}$ is compact if and only if $N_1+N_2$ is closed and $P_{N_1+N_2}-\left(P_{N_1}+\right.$ $\left.P_{N_2}\right)$ is compact.
\end{lemma} 

%We first require the following lemma. 

\begin{lemma} \label{lem:1}
	If $\dim(H^2(\mathbb{D}^n)/(\varphi H^2(\mathbb{D}^n) + \psi H^2(\mathbb{D}^n))) < \infty$, $n > 1$, then $\varphi H^2(\mathbb{D}^n) \cap \psi H^2(\mathbb{D}^n) = \varphi\psi H^2(\mathbb{D}^n)$.
\end{lemma}

\begin{proof}
	It is clear that $\varphi\psi H^2(\mathbb{D}^n) \subseteq \varphi H^2(\mathbb{D}^n) \cap \psi H^2(\mathbb{D}^n)$. It is enough to prove $\varphi H^2(\mathbb{D}^n) \cap \psi H^2(\mathbb{D}^n) \subseteq \varphi\psi H^2(\mathbb{D}^n)$. For arbitrary $f \in \varphi H^2(\mathbb{D}^n) \cap \psi H^2(\mathbb{D}^n)$, write $f = \varphi g = \psi h$, where $g, h \in H^2(\mathbb{D}^n)$. Let
	$$
    K := \frac{g}{\psi} = \frac{h}{\varphi}.
    $$
	Hence, $K$ can be analytically extended to $\mathbb{D}^n \setminus (Z(\varphi) \cap Z(\psi))$, where \(Z(f)\) denotes the zero set of function \(f\). Since $\varphi H^2+\psi H^2$ has finite codimension, there exists $0 < r_0 < 1$ and $\delta > 0$ such that $|\varphi(r\xi)|^2 + |\psi(r\xi)|^2 > \delta$ for all $\xi \in \mathbb{T}^n$ and $r > r_0$ (see Lemma 2.2 in \cite{Wang2025}).  It is easy to check that $Z(\varphi) \cap Z(\psi) \cap \mathbb{D}^n$ is a compact analytic subvariety of $\mathbb{C}^n$. Hence $Z(\varphi) \cap Z(\psi) \cap \mathbb{D}^n$ is a finite set(\cite{Rudin1980} Theorem 14.3.1), which implies that $\mathbb{D}^n \setminus (Z(\varphi) \cap Z(\psi))$ is connected. Hence, we have $K \in H(\mathbb{D}^n)$ by Hartogs's Theorem. If $r > r_0$ we have
	$$|K(r\xi)|^2 = \frac{|h(r\xi)|^2}{|\varphi(r\xi)|^2} = \frac{|g(r\xi)|^2}{|\psi(r\xi)|^2} = \frac{|h(r\xi)|^2 + |g(r\xi)|^2}{|\varphi(r\xi)|^2 + |\psi(r\xi)|^2} < \frac{|h(r\xi)|^2 + |g(r\xi)|^2}{\delta}, \quad \xi \in \mathbb{T}^n.$$
	Hence, $K \in H^2(\mathbb{D}^n)$, that is $\varphi H^2(\mathbb{D}^n) \cap \psi H^2(\mathbb{D}^n) \subseteq \varphi\psi H^2(\mathbb{D}^n)$.
\end{proof}

\subsection{Results}
We say that two functions \( f_1, f_2 \in H^2(\mathbb{D}^2) \) are \emph{separable} if \( f_1 \) depends only on one variable (either \( z_1 \) or \( z_2 \)) and \( f_2 \) depends only on the other. Let \(T=(T_\varphi,T_\psi)\) and \(\Delta_T\) denote the defect operator of \(T=(T_\varphi,T_\psi)\) in the following.

\begin{proposition} \label{lem:2}
	For inner functions $\varphi$ and $\psi$ in $H^\infty(\mathbb{D}^n)$, $n > 1$, the following are equivalent:
	\begin{enumerate}
		\item[(1)] $P_{\mathcal{Q}_\varphi} P_{\mathcal{Q}_\psi}$ is compact;
		\item[(2)] $\Delta_T$ is compact.
	\end{enumerate}
\end{proposition}

\begin{proof}
	(2) $\Rightarrow$ (1) is obviously by Lemma \ref{lem_1}.
    
    (1) \(\Rightarrow\) (2). By Lemma \ref{thm_3}, we know that $P_{\mathcal{Q}_\varphi} P_{\mathcal{Q}_\psi}$ is compact if and only if $\mathcal{Q}_\varphi + \mathcal{Q}_\psi$ is closed and $P_{\mathcal{Q}_\varphi + \mathcal{Q}_\psi} - P_{\mathcal{Q}_\varphi} - P_{\mathcal{Q}_\psi}$ is compact. Hence, by Lemma \ref{lem:1}, we have
	$$
	\begin{aligned}
		&P_{\mathcal{Q}_\varphi + \mathcal{Q}_\psi} - P_{\mathcal{Q}_\varphi} - P_{\mathcal{Q}_\psi}\\  = &P_{\left( \varphi H^2 \cap \psi H^2\right)^{\perp}} - P_{\mathcal{Q}_\varphi} - P_{\mathcal{Q}_\psi} \\
		=&P_{\mathcal{Q}_{\varphi \psi}} - P_{\mathcal{Q}_\varphi} - P_{\mathcal{Q}_\psi} \\
		=&- \Delta_T. 
	\end{aligned} 
	$$
	Hence $\Delta_T$ is compact.
\end{proof}

For the bidisc ($n=2$), we establish the following equivalence.

\begin{theorem}\label{thm_B}
	For nonconstant inner functions $\varphi$ and $\psi$ in $H^{\infty}\left(\mathbb{D}^2\right)$, the following are equivalent:\par 
    \begin{enumerate}
        \item $P_{\mathcal{Q}_{\varphi}} P_{\mathcal{Q}_\psi}$ is compact;
        \item $\Delta_T$ is compact;
        \item $P_{\mathcal{Q}_{\varphi}} P_{\mathcal{Q}_\psi}$ is a finite rank projection;
        \item $\varphi$ and $\psi$ are separated finite Blaschke products.
    \end{enumerate}

\end{theorem}
\begin{proof}
	The equivalence $(1) \Leftrightarrow (2)$ is exactly the result of Proposition \ref{lem:2}. The implication $(3) \Rightarrow (1)$ is trivial.
	It remains to show $(2) \Rightarrow (4)$ and $(4) \Rightarrow (3)$.

	(2)$\Rightarrow$(4). Assume that $\Delta_T$ is compact. By Lemma \ref{lem_1}, the commutator $[T_\varphi, T_\psi^*] = T_\varphi T_{\overline{\psi}} - T_{\overline{\psi}} T_\varphi$ is compact. It is easy to see that $\varphi$ and $\psi$ must be separated since $\varphi$ and $\psi$ are both analytic by Theorem \ref{ding}. 
	
	Without loss of generality, assume $\varphi$ and $\psi$ depend only on $z_1$ and $z_2$ respectively. We have
	$$ \operatorname{Ran}(P_{\mathcal{Q}_\varphi} P_{\mathcal{Q}_\psi}) = \left( H^2(\mathbb{D}_1) \ominus \varphi H^2(\mathbb{D}_1) \right) \otimes \left( H^2(\mathbb{D}_2) \ominus \psi H^2(\mathbb{D}_2) \right), $$
    where $H^2(\mathbb{D}_i)$ denotes the univariate Hardy space depending only on the variable $z_i$ for $i=1,2$.
	
	Since $\Delta_T$ is compact, the equivalence established earlier ensures that $P_{\mathcal{Q}_\varphi} P_{\mathcal{Q}_\psi}$ is also compact. This occurs if and only if both $H^2(\mathbb{D}_1) \ominus \varphi H^2(\mathbb{D}_1)$ and $H^2(\mathbb{D}_2) \ominus \psi H^2(\mathbb{D}_2)$ are finite-dimensional, which implies that $\varphi$ and $\psi$ are separated finite Blaschke products.
	
	(4)$\Rightarrow$(3). . Since $\varphi$ and $\psi$ are separated, then $T_{\psi}T_{\varphi}^* = T^*_{\varphi} T_{\psi}$. This implies that  $P_{\mathcal{Q}_{\varphi}}  P_{\mathcal{Q}_\psi}$ is a projection. Hence we have 
    $$
    \begin{aligned}
        \operatorname{Ran}P_{\mathcal{Q}_\varphi}P_{\mathcal{Q}_\psi}&=\operatorname{Ran}P_{\mathcal{Q}_\varphi}\cap\operatorname{Ran}P_{\mathcal{Q}_{\psi}}\\
        &=\mathcal{Q}_\varphi\cap\mathcal{Q}_{\psi}\\
        &=\left( H^2(\mathbb{D}_1) \ominus \varphi H^2(\mathbb{D}_1) \right) \otimes \left( H^2(\mathbb{D}_2) \ominus \psi H^2(\mathbb{D}_2) \right).
    \end{aligned}
    $$
    Therefore, we conclude that \(P_{\mathcal{Q}_\varphi}P_{\mathcal{Q}_{\psi}}\) is a finite rank projection.
\end{proof}

Finally, for a polydisc of dimension strictly greater than 2, the condition becomes even more restrictive, implying that a compact product of such projections must be trivial:
\begin{theorem}\label{thm_C}
	For nonconstant inner functions $\varphi$ and $\psi$ in $H^{\infty}\left(\mathbb{D}^n\right), n>2$, the following are equivalent:\par 
    \begin{enumerate}
        \item $P_{\mathcal{Q}_{\varphi}} P_{\mathcal{Q}_\psi}$ is compact;
        \item $\Delta_T$ is compact;
        \item $P_{\mathcal{Q}_{\varphi}} P_{\mathcal{Q}_\psi}=0$.
    \end{enumerate}
    
\end{theorem}
\begin{proof}
	As before, $(1) \Leftrightarrow (2)$ follows directly from Proposition \ref{lem:2}, and $(3) \Rightarrow (1)$ is trivial. We only need to prove $(2) \Rightarrow (3)$.
	
    Since $\Delta_T$ is compact, Lemma \ref{lem_1} implies that $[T_\varphi^*, T_\psi]$ is compact. By Theorem \ref{ding}, the compactness of $[T_\varphi^*, T_\psi]$ is equivalent to $\varphi$ and $\psi$ being separable, which in turn implies that $P_{\mathcal{Q}_{\varphi}} P_{\mathcal{Q}_\psi}$ is a projection. 

    Moreover, according to the result in \cite{GuoWang2007}, the compactness of $\Delta_T$ also implies that $\varphi H^2(\mathbb{D}^n) + \psi H^2(\mathbb{D}^n)$ has finite codimension. For $n \geq 3$, the classical result of Ahern and Clark in \cite{Clark} states that the invariant subspace generated by $\varphi$ and $\psi$ is either the whole space $H^2(\mathbb{D}^n)$ or has infinite codimension. Consequently, we must have 
   \[
   \varphi H^2(\mathbb{D}^n) + \psi H^2(\mathbb{D}^n) = H^2(\mathbb{D}^n),
   \]
    which means $\mathcal{Q}_\varphi \cap \mathcal{Q}_\psi = \{0\}$. Hence we conclude that $P_{\mathcal{Q}_{\varphi}} P_{\mathcal{Q}_\psi} = 0$.

\end{proof}
\subsection{Remark}
Debnath et al. proved that the product of two inner projections is a finite-rank projection if and only if $\varphi$ and $\psi$ are separated finite Blaschke products \cite{Debnath2024}. This result answers Douglas's original question in \cite{Bickel2016}. Their proof relies on Sarkar's results of partially isometric Toeplitz operators \cite{Sarkar2018}. We arrive at the same conclusion independently by characterizing the compactness of the product of inner projections. In the multivariate setting, we further observe that compactness of the product of inner projections reduces to the operator being of finite rank.

In \cite{Sarkar2018}, the authors also characterized when \([P_{\mathcal{Q}_{\varphi}},P_{\mathcal{Q}_{\psi}}] = 0\). One may ask the following generalization of Douglas's original question: 
\begin{center}
    When \([P_{\mathcal{Q}_{\varphi}},P_{\mathcal{Q}_{\psi}}]\) is compact on \(H^2(\mathbb{D}^n)\) for \(n\geq2\)?
\end{center}
We suspect the answer to this question may reduce to the case where \([P_{\mathcal{Q}_{\varphi}}, P_{\mathcal{Q}_{\psi}}] = 0\). However, the problem remains open.

\medskip
\bigskip

\subsection*{Acknowledgment}
Y. Lu was supported by the National Natural Science Foundation of China (Grant No.12031002), Y. Yang was supported by the National Natural Science Foundation of China (Grant No.12471117), C. Zu was supported by the National Natural Science Foundation of China (Grant No.12401151), and the Postdoctoral Researcher Foundation of China (Grant No.GZB20240100).

%\subsection*{Author Contribution Statement}
%All authors contributed equally to the research, analysis, and preparation of the manuscript. Each author approved the final version of the paper and agrees to be accountable for all aspects of the work.
%
\subsection*{Conflict of interest}
The authors have no conflict of interest to declare that are relevant to the content of this article. 
\subsection*{Data availability statement}
No data, models, or code were generated or used for the research described in the article.%Data sharing is not applicable to this article as no new data were created or analyzed in this study.


\begin{thebibliography}{99}
\bibitem{Clark}
P.R. Ahern and D.N. Clark, Invariant subspaces and analytic continuation in several variables, J. Math. Mech. {\bf 19} (1969/70), 963--969.

\bibitem{ACS1978}
S. Axler, S.Y. A. Chang, and D. Sarason, \emph{Product of Toeplitz operators}, Integral Equations Operator Theory \textbf{1} (1978), 285--309.

\bibitem{Beurling1948}
A. Beurling, \emph{On two problems concerning linear transformations in Hilbert space}, Acta Math. \textbf{81} (1948), 17 pp.

\bibitem{Berger1978}
C.~A. Berger, L.~A. Coburn and A. Lebow, Representation and index theory for $C\sp*$-algebras generated by commuting isometries, J. Functional Analysis {\bf 27} (1978), no.~1, 51--99.

\bibitem{Bickel2016}
K. Bickel and C. Liaw, Properties of Beurling-type submodules via Agler decompositions, J. Funct. Anal. {\bf 272} (2017), no.~1, 83--111.

\bibitem{Bottcher2010}
A. Böttcher and I. M. Spitkovsky, \emph{A gentle guide to the basics of two projections theory}, Linear Algebra and its Applications \textbf{432} (2010), no. 6, 1412--1459.

\bibitem{Chu2015}
C. Chu, \emph{Compact product of Hankel and Toeplitz operators on the Hardy space}, Indiana University Mathematics Journal (2015), 973--982.

\bibitem{Davis1958}
C. Davis, \emph{Separation of two linear subspaces}, Acta Scientiarum Mathematicarum (Szeged) \textbf{19} (1958), 172--187.

\bibitem{Debnath2024}
R. Debnath, D. K. Pradhan, and J. Sarkar, \emph{Pairs of inner projections and two applications}, J. Funct. Anal. \textbf{286} (2024), no. 2, Paper No. 110216, 26 pp.

\bibitem{Sarkar2018}
K.~D. Deepak, D.~K. Pradhan and J. Sarkar, Partially isometric Toeplitz operators on the polydisc, Bull. Lond. Math. Soc. {\bf 54} (2022), no.~4, 1350--1362.

\bibitem{Ding2003}
X.~H. Ding, \emph{Products of Toeplitz operators on the polydisk}, Integral Equations Operator Theory {\bf 45} (2003), no.~4, 389--403.

\bibitem{Dixmier1948}
J. Dixmier, \emph{Position relative de deux variétés linéaires fermées dans un espace de Hilbert}, Revue Scientifique \textbf{86} (1948), 387--399.

\bibitem{Douglas1973}
R. G. Douglas, \emph{Banach algebra techniques in the theory of Toeplitz operators}, CBMS, vol. 15, A.M.S, Providence, 1973.

\bibitem{Garnett1981}
J. B. Garnett, \emph{Bounded Analytic Functions}, Academic Press, New York, 1981.

\bibitem{Guillory1984}
P.~B. Gorkin, Singular functions and division in $H\sp{\infty }+C$, Proc. Amer. Math. Soc. {\bf 92} (1984), no.~2, 268--270.

\bibitem{Guillory1981}
C.~J. Guillory and D.~E. Sarason, Division in $H\sp{\infty }+C$, Michigan Math. J. {\bf 28} (1981), no.~2, 173--181.

\bibitem{Guo2004}
K.~Y. Guo, Defect operators, defect functions and defect indices for analytic submodules, J. Funct. Anal. {\bf 213} (2004), no.~2, 380--411.

\bibitem{GuoWang2007}
K.Y. Guo and P.H. Wang, \emph{Defect operators and Fredholmness for Toeplitz pairs with inner symbols}, J. Operator Theory \textbf{58} (2007), no. 2, 251--268.

\bibitem{GuoWang20072}
K.Y. Guo and P.H. Wang, Essentially normal Hilbert modules and $K$-homology. III. Homogenous quotient modules of Hardy modules on the bidisk, Sci. China Ser. A {\bf 50} (2007), no.3, 387--411.

\bibitem{GuoYang}
K.~Y. Guo and R. Yang, The core function of submodules over the bidisk, Indiana Univ. Math. J. {\bf 53} (2004), no.~1, 205--222.

\bibitem{Halmos1969}
P. R. Halmos, \emph{Two subspaces}, Transactions of the American Mathematical Society \textbf{144} (1969), 381--389.

\bibitem{Hoffman1967}
K.~M. Hoffman, Bounded analytic functions and Gleason parts, Ann. of Math. (2) {\bf 86} (1967), 74--111.

\bibitem{Izuchi2002}
K.~J. Izuchi, Outer and inner vanishing measures and division in $H^\infty+C$, Rev. Mat. Iberoamericana {\bf 18} (2002), no.~3, 511--540.

\bibitem{Rudin1969}
W. Rudin, \emph{Function theory in polydiscs}, Benjamin, New York, 1969.

\bibitem{Rudin1980}
W. Rudin, \emph{Function Theory in the unit ball of $\mathbb{C}^n$}, Springer Verlag, 1980.

\bibitem{Sarason1967}
D. Sarason, \emph{Generalized interpolation in $H^\infty$}, Transactions of the American Mathematical Society \textbf{127} (1967), no. 2, 179--203.

\bibitem{Wang2025}
P.~H. Wang and Z. Zhu, Fredholm index of Toeplitz pairs with $H^\infty$ symbols, Canad. Math. Bull. {\bf 68} (2025), no.~1, 166--176.

\bibitem{Yang2005}
R. Yang, The core operator and congruent submodules, J. Funct. Anal. {\bf 228} (2005), no.~2, 469--489.

\bibitem{YangRW}
R. Yang, Operator theory in the Hardy space over the bidisk. II, Integral Equations Operator Theory {\bf 42} (2002), no.~1, 99--124.

\bibitem{BiswasSarkar26}
R. Biswas and S. Sarkar, Essentially commuting projections onto shift-invariant subspaces, arXiv:2606.25649 [math.FA], 2026, \url{https://arxiv.org/abs/2606.25649}.

\end{thebibliography}
\end{document}